\documentclass[10pt]{amsart}
\usepackage[centertags]{amsmath}
\usepackage{amsxtra}
\usepackage[top=1in,bottom=1in,left=1in,right=1in]{geometry}
\usepackage{amssymb}
\usepackage{amsthm}
\usepackage{enumerate}
\usepackage{newlfont}
\newtheorem{theorem}{Theorem}
\newtheorem{definition}{Definition}[section]
\newtheorem{lemma}{Lemma}
\newtheorem{pro}[definition]{Proposition}
\newtheorem{corollary}[definition]{Corollary}

\newtheorem{exam}[definition]{Example}
\numberwithin{equation}{section}

\usepackage{amsfonts}
\usepackage{mathrsfs}
\usepackage{amsmath}
\usepackage{amscd}
\usepackage[all,cmtip]{xy}

\usepackage{dsfont}
\usepackage{tikz}
\usepackage{tikz}
\usepackage{pgflibraryarrows}
\usepackage{pgflibrarysnakes}

\begin{document}

\title{\textbf{Frobenius Splitting of Projetive Toric Bundles}}
\date{April 8, 2014}
\author{He Xin}
\email{hexin1733@gmail.com}

\maketitle
\begin{abstract} We give several mild conditions on a toric bundle on a nonsingular toric variety under which the projectivization of the toric bundle is Frobenius split.

\end{abstract}

\section{Introduction}

Let $X$ be a complete toric variety defined over an algebraically closed field of characteristic $p>0$. In \cite[Quesiton 7.6]{hmp2010}, the authors ask when the projectivization of a toric bundle on $X$ is Frobenius split.

\smallskip Recall that to any toric bundle $\mathcal{E}$ of rank $r$ on a toric variety $X(\Sigma)$ one can associate uniquely a compatible family of decreasing filtrations $\{E^{\alpha}(i)\}_{\alpha\in|\Sigma|}$ of a fixed vector space $E$ of dimension $r$ indexed by the finite set $|\Sigma|$ consisting of primitive vectors in $\Sigma$. For each toric bundle $\mathcal{E}$ and each $\alpha\in|\Sigma|$, we introduce a finite subset of integers $I^{\alpha}(\mathcal{E})\subset \mathds{Z}$, which is defined by $$i\in I^{\alpha}(\mathcal{E})\Leftrightarrow E^{\alpha}(i+1)\varsubsetneq E^{\alpha}(i),$$
\noindent and two numbers
 $$n_{\mathrm{max}}^{\alpha}(\mathcal{E})=\max\{i\in I^{\alpha}(\mathcal{E})\},\  \  n_{\mathrm{min}}^{\alpha}(\mathcal{E})=\min\{i\in I^{\alpha}(\mathcal{E})\}.$$

\noindent The \emph{Klyachko length} of a toric bundle $\mathcal{E}$ is defined to be
$$\mathrm{Kl}(\mathcal{E})=\max_{\alpha\in|\Sigma|}\{n_{max}^{\alpha}(\mathcal{E})-n_{min}^{\alpha}(\mathcal{E})\}.$$
  In this note we will show the following
\begin{theorem}\label{mt} Let $X$ be a smooth toric variety and $\mathcal{E}$ be a rank two toric bundle on $X$, then $\mathbb{P}(\mathcal{E})$ is Frobenius split provided $p>\mathrm{Kl}(\mathcal{E})$ and $n_{\mathrm{max}}^{\alpha}(\mathcal{E})\neq n_{\mathrm{min}}^{\alpha}(\mathcal{E})$ for all $\alpha\in|\Sigma|$.
\end{theorem}

For a toric bundle $\mathcal{E}$ on a complete toric variety $X$, there is an induced action of the torus $T$ on the vector space $H^0(X,\mathcal{E})$ and the $\chi$ isotopy component is denoted by $H^0(X,\mathcal{E})_{\chi}$. Let $\mu$ be the zero character, we will also prove the following

\begin{theorem}\label{mt2}  Let $X$ be a smooth complete toric variety and $\mathcal{E}$ be a toric bundle of rank $r$ on $X$. Then $\mathbb{P}(\mathcal{E})$ is Frobenius split if and only if there exists an FS-vector in $H^0(X,S^{r(p-1)}\mathcal{E}\otimes\det\mathcal{E}^{1-p}\otimes\omega_X^{1-p})_{\mu}$, which is regarded as a subspace of $S^{r(p-1)}E$.
\end{theorem}

We will work throughout over an algebraically closed field $k$ of characteristic $p>0$.

\section{Preliminaries}

The objective of this section is twofold. The first is to provide two criteria for Frobenius splitting of a projective bundle over a smooth complete variety and the second is to review some facts on toric bundles.
\subsection{Criteria for Frobenius Splitting of Projective Bundles}\hfill

\smallskip Given an algebraic variety $X$ over $k$, it is said to be \emph{Frobenius split} or \emph{F-split} if the morphism \begin{equation}\mathcal{O}_X\rightarrow F_{X*}\mathcal{O}_X\label{fr}\end{equation}
\noindent defined by sending a section to its $p$-th power is split as a homomorphism of $\mathcal{O}_X$-modules.

\smallskip Suppose $X$ is \emph{F}-split, given a vector bundle $\mathcal{E}$ on $X$, next we will give two conditions on $\mathcal{E}$ under which the projective bundle $\mathbb{P}(\mathcal{E})$ is \emph{F}-split.

\subsubsection*{Criterion A}\hfill

\smallskip Let $\mathbb{V}(\mathcal{E})$ be the affine scheme $\mathrm{Spec}\bigoplus_{n\geq0}
S^n\mathcal{E}$ over $X$, then we have the following

\begin{pro}\label{pv}\cite[Example 1.1.10(3)]{kumarbrion} Let $\mathcal{E}$ be a vector bundle on $X$ such that there exists a section of the projection $\mathbb{V}(\mathcal{E})\rightarrow X$. Then if $\mathbb{V}(\mathcal{E})$ is Frobenius split, so is $\mathbb{P}(\mathcal{E})$.
\end{pro}

Since $\mathbb{P}(\mathcal{E})$ is isomorphic to $\mathbb{P}(\mathcal{E}\otimes \mathcal{L})$ for any line bundle $\mathcal{L}$ as a $k$-variety, we can always assume there exists a section of the projection $\mathbb{V}(\mathcal{E})\rightarrow X$.

\smallskip Next we assume $X$ is smooth, then the affine bundle $\mathbb{V}(\mathcal{E})$ over $X$ is \emph{F}-split locally since any smooth affine variety is \emph{F}-split \cite[Prop. 1.16]{kumarbrion}. To be more explicit, let $X$ be a smooth affine variety,
$\varphi:F_{X*}\mathcal{O}_X\rightarrow\mathcal{O}_X$ be a Frobenius splitting, and $\mathcal{E}$ be a trivializable vector bundle on $X$
with basis $e_i, 1\leq i\leq n$. Then one sees easily the following morphism defines a Frobenius splitting of $\mathbb{V}(\mathcal{E})$.
\begin{equation}\label{fs}
\phi(ae_1^{\otimes s_1}\otimes\cdots \otimes e_r^{\otimes s_r})=
\begin{cases}\varphi(a)e_1^{\otimes \frac{s_1}{p}}\otimes\cdots\otimes e_r^{\otimes \frac{s_r}{p}}, & \text{if $p\,|\,s_i$ for all $1\leq i\leq r$};\\
             0, & \text{otherwise}.
\end{cases}
\end{equation}

For $r\geq2$ one can check easily the Frobenius splitting of $\mathbb{V}(\mathcal{E})$ defined above depends on the choice of the basis of $\mathcal{E}$. In general, the constraints on the transition matrix between two different basis to ensure the coincidence of the two associated Frobenius splittings are very complicated to write down. However, when the rank of $\mathcal{E}$ is two we have the following

\begin{lemma}\label{le} Let
$\{f_1,f_2\}$ and $\{e_1,e_2\}$ be two basis of $\mathcal{E}$ and $(a_{ij})_{2\times2}$ be the transition matrix from $\{e_1,e_2\}$ to $\{f_1,f_2\}$.
Then the Frobenius splitting given by (\ref{fs}) for $\mathbb{V}(E)$ associated to the two bases $\{f_1,f_2\}$ and $\{e_1,e_2\}$ coincide iff the invertible matrix $(a_{ij})$ satisfies $\varphi(a^s_{11}a^{p-s}_{21})=\varphi(a^s_{12}a^{p-s}_{22})=0$ for all $1\leq s\leq p-1$.

\end{lemma}
\begin{proof} Let $\phi_f$ and $\phi_e$ be the Frobeniuis splitting associated to the basis $\{e_1,e_2\}$ and $\{f_1,f_2\}$ as given in (\ref{fs}) respectively. Then the coincidence of $\phi_f$ and $\phi_e$ is equivalent to
\begin{equation}\label{ef}\phi_f(f_1^{\otimes r}\otimes f_2^{\otimes s})=\phi_e((a_{11}e_1+a_{12}e_2)^{\otimes r}\otimes(a_{21}e_1+a_{22}e_2)^{\otimes s})
\end{equation}

\noindent for all pairs of nonnegative integers $(r,s)$. Let $s=p-r$, by comparing the coefficients of $e_1$ and $e_2$ on the two sides one obtains easily the necessity of the condition in the lemma.

\smallskip Assuming the condition in the lemma, now we prove the equality (\ref{ef}). If $p\nmid r+s$ one can see easily both sides the above equality are 0. Now we assume $r+s=pm$, if $r\,|\,p$, then the equality (\ref{ef}) follows easily from the definition (\ref{fs}). If $r\nmid p$, let $r=pm_1+q_1$, $s=pm_2+q_2$ such that $1\leq q_1,q_2\leq p-1$, then $q_1+q_2=p$. The right side of (\ref{ef}) can be written as
$$(a_{11}e_1+a_{12}e_2)^{\otimes m_1}\otimes(a_{21}e_1+a_{22}e_2)^{\otimes m_2}\otimes\phi_e((a_{11}e_1+a_{12}e_2)^{\otimes q_1}\otimes(a_{21}e_1+a_{22}e_2)^{\otimes q_2}).$$
\noindent Then by the condition in the lemma one can see easily $\phi_e((a_{11}e_1+a_{12}e_2)^{\otimes q_1}\otimes(a_{21}e_1+a_{22}e_2)^{\otimes q_2})$ is 0 hence the lemma is proved.

\end{proof}

Based on the lemma above, now we can give our first criterion.

\begin{pro}\label{ca} Let $\mathcal{E}$ be vector bundle of rank two over a smooth algebraic variety $X$ which is F-split, then the $\mathbb{P}^1$-bundle $\mathbb{P}(\mathcal{E})$ is also F-split, if there exists a representation of $\mathcal{E}$ as an element of $H^1(X,\mathcal{GL}_2)$ such that all the transition matrices given by this representation satisfy the conditions of lemma \ref{le}.
\end{pro}

\subsubsection*{Criterion B}\hfill

\smallskip Firstly note that a splitting of (\ref{fr}) yields
a nonzero global section of the sheaf $\mathscr{H}om(F_{X*}\mathcal{O}_X,\mathcal{O}_X)$. Now we introduce the evaluation map
\begin{equation}\label{evl}  \epsilon:\mathscr{H}om(F_{X*}\mathcal{O}_X,\mathcal{O}_X)\rightarrow\mathcal{O}_X,\,\,\,\varphi\mapsto\varphi(1),
\end{equation}
\noindent then a splitting of (\ref{fr}) is equivalent to the existence of a global section $\varphi$ of $\mathscr{H}om(F_{X*}\mathcal{O}_X,\mathcal{O}_X)$ such that $\epsilon(\varphi)=1$.

\smallskip The reformulation of Frobenius splitting given above is the starting point of our second criterion. Suppose we are given a proper morphism $f:Y\rightarrow X$ of algebraic varieties such that $f^{\#}:\mathcal{O}_X\rightarrow f_*\mathcal{O}_Y$ is an isomorphism, then we have an induced morphism \begin{equation}\label{iso2}f_*\mathscr{H}om(F_{Y*}\mathcal{O}_Y,\mathcal{O}_Y)\rightarrow\mathscr{H}om(F_{X*}\mathcal{O}_X,\mathcal{O}_X).\end{equation}
\noindent The morphism above obviously commutes with the evaluation map (\ref{evl}). If the this morphism is surjective, then it follows that the \emph{F}-splitness of $X$ implies that of $Y$, c.f. \cite[1.3.E(7)]{kumarbrion}.

\medskip Next based on a result of duality theory, we will simplify the checking of the surjectivity of (\ref{iso2}). Firstly, by definition of the upper shriek functor \cite[Chapter 2, exe. 6.10]{hartshorne1977algebraic}, we have the following isomorphism \begin{equation}\label{shriek}\mathscr{H}om(F_{X*}\mathcal{O}_X,\mathcal{O}_X)\cong F_{X*}\mathscr{H}om(\mathcal{O}_X,F^!\mathcal{O}_X).\end{equation}

From now on we assume $X$ is a smooth complete variety and $f$ is also a smooth morphism, which includes in particular the case when $Y$ is a projective bundle over $X$. Then by \cite{hartshorne1966} we have an isomorphism
\begin{equation}\label{gd}F^!\mathcal{O}_X\cong\mathscr{H}om(F^*\omega_X,\omega_X)\cong \omega^{1-p}_{X}.\end{equation}
\noindent Therefore, by combining (\ref{shriek}) and (\ref{gd}) we have the following isomorphism of sheaves of $\mathcal{O}_X$-modules
\begin{equation}\mathscr{H}om(F_{X*}\mathcal{O}_X,\mathcal{O}_X)\cong F_{X*}\omega^{1-p}_{X},\label{iso1}\end{equation}

\noindent whence an induced isomorphism from the space of global sections of $\mathscr{H}om(F_{X*}\mathcal{O}_X,\mathcal{O}_X)$ to $H^0(X,\omega^{1-p}_X)$.

The above argument also applies to $Y$. Furthermore, the isomorphism (\ref{iso1}) is canonical thus from (\ref{iso2}) we have an induced morphism
\begin{equation}\label{iso3}f_*F_{Y*}\omega^{1-p}_Y\rightarrow F_{X*}\omega^{1-p}_X.
\end{equation}
\noindent In particular, we have an induced morphism on global sections
\begin{equation}\label{gs}\pi: H^0(Y,\omega_Y^{1-p})\rightarrow H^0(X,\omega_X^{1-p})
\end{equation}
By our earlier argument, if $\pi$ is surjective, the \emph{F}-splitness of $X$ implies that of $Y$ \cite[1.3E(7)]{kumarbrion}.

\medskip The morphism (\ref{iso1}) can be derived in a more direct way, based on which (\ref{iso3}) and (\ref{gs}) will become more explicit. First note that we have obviously
\begin{equation}F_{X*}\omega^{1-p}_X\cong F_{X*}\mathscr{H}om(F_X^*\omega_X,\omega_X)\cong\mathscr{H}om(\omega_X,F_{X*}\omega_X).\label{iso1'}\end{equation}
\noindent We will define an isomorphism from $\mathscr{H}om(\omega_X,F_{X*}\omega_X)$ to $\mathscr{H}om(F_{X*}\mathcal{O}_X,\mathcal{O}_X)$ and our new description of (\ref{iso1}) will be defined to be the composite of (\ref{iso1'}) with this isomorphism. For that purpose, we need to use the trace map $$\tau:F_{X*}\omega_X\rightarrow\omega_X.$$

\noindent Recall that $\tau$ is the composite of the projection $F_{X*}\omega_X\rightarrow\mathcal{H}^n(F_{X*}\Omega^{\bullet}_{X/k})$ with the Cartier isomorphism
$$\mathcal{H}^n(F_{X*}\Omega^{\bullet}_{X/k})\xrightarrow{\approx}\Omega^n_{X/k},$$
\noindent where $n=\dim X$.  The trace map can be written explicitly in terms of local coordinates as follows.

\begin{lemma}\cite[Proposition 1.3.6]{kumarbrion}\label{tr} Let $X$ be a nonsingular variety of dimension $n$ and $x_1,\cdots,x_n$ be a set of coordinates at a point $x$ of $X$, then the trace map $\tau$ is given by
$$\tau(fdx_1\wedge\cdots\wedge dx_n)=Tr(f)dx_1\wedge\cdots\wedge dx_n,$$
\noindent where $f=\sum_{\mathbf{i}}f_{\mathbf{i}}x^{\mathbf{i}}\in\mathcal{O}_{X,x}\subset k[[x_1,\cdots,x_n]]$ and $Tr(f):=\sum_{\mathbf{i}}f^{\frac{1}{p}}_{\mathbf{i}}x^{\mathbf{j}}$ with the summation taken over all multi-index $\mathbf{i}$ such that $\mathbf{i}=\mathbf{p-1}+p\mathbf{j}$.

\end{lemma}

As we mentioned above, one can prove the isomorphism (\ref{iso1}) is nothing but the composite of (\ref{iso1'}) and a morphism $\iota$ given by the following
\begin{lemma}\label{le3}\cite[Proposition 1.3.7]{kumarbrion} The morphism
\begin{equation}\label{iota}\iota:\mathscr{H}om(\omega_X,F_{X*}\omega_X)\rightarrow \mathscr{H}om(F_{X*}\mathcal{O}_X,\mathcal{O}_X)\end{equation}
defined by $\iota(\psi)(f)\omega=\tau(f\psi(\omega))$, where $f$ is a section of $\mathcal{O}_X$, $\omega$ is a local generator of $\omega_X$, is an isomorphism.
\end{lemma}

\subsection{Toric Bundles}\hfill

\smallskip In this part we will recall some facts on toric bundles, among which the \emph{Klaychko data} of a toric bundle will play a key role later in the proof of our theorems.

A toric variety is a normal algebraic variety on which there is an action of a torus such that over an open subset of this variety the action is free. Any affine toric variety can be constructed as follows. Let $N$ be a free abelian group of rank $n$, $\sigma\subseteq N\otimes_{\mathds{Z}}\mathds{R}$ be a strongly convex rational polyhedral cone, $\sigma\spcheck$ be the set of vectors in $N\spcheck\otimes_{\mathds{Z}}\mathds{R}$ which take nonnegative values on $\sigma$. Then $S_{\sigma}=\sigma\spcheck\cap N\spcheck$ is commutative semigroup and $U_{\sigma}=\mathrm{Spec}\,k[S_{\sigma}]$ is an affine toric variety. Given a general toric variety of dimension $n$, there exists a fan $\Sigma$ in $N\otimes_{\mathds{Z}}\mathds{R}$ from which one can obtain the toric variety by firstly constructing the affine toric varieties $U_{\sigma}$ for each cone $\sigma\subseteq\Sigma$ as above and then glueing $\{U_{\sigma}\}_{\sigma\in \Sigma}$.

\smallskip A toric bundle on $X$ is defined to be a vector bundle $\pi:\mathscr{E}\rightarrow X$ endowed with an action of $T$ which is compatible with the action of $T$ on $X$. In the sequel, by a toric bundle we mean the associated locally free sheaf.

\subsubsection*{$\mathrm{2.2.0.}$ Toric bundles over Affine Toric varieties}\hfill

\smallskip Given an affine toric variety $U_{\sigma}$, one can show \cite[Proposition 2.1.1]{Klyachko1990} a toric bundle $\mathcal{E}$ on $U_{\sigma}$ is a direct sum of toric line bundles whose underlying bundles are trivial. In other words, we have the following isomorphism of toric bundles
\begin{equation}\label{dec0}\mathcal{E}\cong\bigoplus_{\smash\chi_i\in\widehat{T},1\leq i\leq n}\chi_i\mathcal{O}_{U_{\sigma}}.\end{equation}

\noindent The underlying line bundle of $\chi_i\mathcal{O}_{U_{\sigma}}$ is $\mathcal{O}_{U_{\sigma}}$ and the action of $T$ is given $t.e_{\sigma}=\chi_i(t)e_{\sigma}$, where $e_{\sigma}$ is the unit of $\mathcal{O}_{U_{\sigma}}$. The inverse images in $\mathcal{E}$ of the units in the direct summands on the right side constitute a basis and we will denote it by $\{e_{\sigma}\}$. This basis will be called an \emph{eigen-basis} of $\mathcal{E}$, and the characters $\{\chi_i\}$ will be called \emph{eigen-characters} of $\mathcal{E}$.

\smallskip The eigen-characters appearing in the decomposition (\ref{dec0}) might not be uniquely determined, for instance any two toric line bundles are isomorphic on a torus. Indeed, can also show loc.cit. the toric bundle structure is uniquely determined by the induced representation of $T_{\sigma}$ on $\mathcal{E}_x$, where $x$ is a point of the unique closed orbit of $U_{\sigma}$ and $T_{\sigma}\subseteq T$ be the stabilizer subgroup of $x$. Moreover, the induced representation of $T_{\sigma}$ on the fiber $\mathcal{E}_x$ can be extended to a representation of $T$ and we take one such representation $\varphi_{\sigma}$.  Then the image of $\{e_{\sigma}\}$ in $\mathcal{E}_x$ are exactly the eigenvectors of the representation $\varphi_{\sigma}$.

\subsubsection{Klyachko Data of a Toric Bundle}\hfill

\smallskip The \emph{Klyachko data} of a toric bundle is a vector space endowed with a family of compatible filtrations as described in the theorem below.
\begin{theorem}\cite[Theorem 2.2.1]{Klyachko1990}\label{kl} To give a toric bundle of rank $r$ on a toric variety $X(\Sigma)$ is equivalent to give a vector space $E$ of dimension $r$ on which there exists a family of compatible filtrations $\{E^{\alpha}(i)\}_{\alpha\in|\Sigma|}$ of $E$ indexed by the primitive vectors of $\Sigma$ satisfying the following condition.

\smallskip For each cone $\sigma\in\Sigma$, there exists a decomposition
\begin{equation}\label{dec}E\cong\bigoplus_{\chi\in \widehat{T}_{\sigma}}E^{[\sigma]}(\chi),\end{equation}
\noindent where $\widehat{T}_{\sigma}$ is the group of characters of the stabilizer subgroup $T_{\sigma}\subseteq T$ of any point in the unique closed orbit of $U_{\sigma}$, such that
$$E^{\alpha}(i)\cong\bigoplus_{\langle\chi,\alpha\rangle\geq i}E^{[\sigma]}(\chi)$$
\noindent for any $\alpha\in|\sigma|$ and $i\in \mathds{Z}$.
\end{theorem}

Next we only recall briefly how to obtain the Klyachko data of a toric bundle. Firstly we assume $X$ is affine. Let $x$ be a point in the unique closed orbit of $X$, then we have an induced representation of $T_{\sigma}$, where $T_{\sigma}\subseteq T$ is the stabilizer subgroup of $x$. Now we extend this representation to a representation of $T$ on $\mathcal{E}_x$ and let $E_{\chi}$ be the $\chi$-isotropic subspace of this representation. Then for any $\alpha\in|\Sigma|$
\begin{equation}\label{eri}E^{\alpha}(i)=\bigoplus_{\langle\chi,\alpha\rangle\geq i}E_{\chi}.\end{equation}

\smallskip For a toric bundle $\mathcal{E}$ on a general toric variety $X$, one can choose a family of open affine sub toric varieties $\{U_{\sigma}\}$ covering $X$. Then define a representation $\varphi_{\sigma}$ of $T$ on $\mathcal{E}_{x_{\sigma}}$ as in section 2.2.0., where $x_{\sigma}$ is a point of the closed orbit of $U_{\sigma}$. Then one can obtain a filtration as above for $\alpha\in|\sigma|=\sigma\cap|\Sigma|$ hence for all $\alpha\in|\Sigma|$. One can prove the filtration $\{E^{\alpha}(i)\}$ obtained in this way is independent on the choice of $\sigma$.

\subsubsection{A Toric Bundle as a 1-coycle}\hfill

\smallskip Recall that any vector bundle of rank $n$ on an algebraic variety can be represented by a 1-cocycle of the sheaf of functions with values in $\mathcal{GL}_n$. Such representations can be particularly simple for a toric bundle, which has been worked out by Kaneyama in \cite{Kaneyama1975}.

\begin{pro}\cite[Theorem 4.2]{Kaneyama1975}\label{ka} Let $X=X(\Sigma)$ be a toric variety and $\mathcal{E}$ be a toric bundle on $X$. Then for each open affine sub toric variety $U_{\sigma}$, $\sigma\in\Sigma$, there exists a basis $\{e_{\sigma}\}$ of $\mathcal{E}|_{U_{\sigma}}$ such that for two cones $\sigma, \tau\in\Sigma$, the transition matrix of $\mathcal{E}$ over $U_{\sigma}\cap U_{\tau}$ with respect to $\{e_{\sigma}\}$, $\{e_{\tau}\}$
can be written as
\begin{equation}\label{tm}\mathrm{Diag}(\chi^1_{\sigma},\cdots,\chi^r_{\sigma})P_{\sigma\tau}
\mathrm{Diag}(\chi^1_{\tau},\cdots,\chi^r_{\tau})^{-1},
\end{equation}
\noindent where
\begin{enumerate}
\item
$\chi^i_{\sigma}$ (resp. $\chi^i_{\tau}$), $1\leq i\leq r$, are the characters of $T$;
\item
$P_{\sigma\tau}\in GL(r,k)$ and
\item
 the set of matrices $\{P_{\sigma\tau}\}_{\sigma, \tau\in\Sigma}$ satisfy the usual cocycle conditions.
\end{enumerate}
\end{pro}
\begin{proof} Since a toric bundle on an affine toric variety is a direct sum of toric line bundles (\ref{dec0}), we have
\begin{equation}\label{dec1}\mathcal{E}|_{U_{\sigma}}\cong\bigoplus_{1\leq i\leq r}{\chi^i_{\sigma}}\mathcal{O}_{U_{\sigma}},\end{equation}

\noindent where $\chi^i_{\sigma}\in\widehat{T}, 1\leq i\leq r$. As in section 2.2.0, we take $\{e_{\sigma}\}$ to be the eigen-basis provided by (\ref{dec1}). Moreover, we take $x_{\sigma}$ to be a closed point of the unique closed orbit of $U_{\sigma}$, then we have an induced representation of $T_{\sigma}$ on $\mathcal{E}_{x_{\sigma}}$ and take $\varphi_{\sigma}$ to be an extension of this representation to $T$.

Let $f_{\sigma\tau}$ be the transition matrix of $\mathcal{E}$ over $U_{\sigma}\cap U_{\tau}$ from the basis $\{e_{\tau}\}$ to $\{e_{\sigma}\}$. Since $U_{\sigma}\cap U_{\tau}$ is a toric variety with an action of $T$, for any $t\in T$ and $x\in U_{\sigma}\cap U_{\tau}$ we have
\begin{equation}\label{tm1}
f_{\sigma\tau}(tx)=\varphi_{\sigma}(t)f_{\sigma\tau}(x)\varphi_{\tau}(t)^{-1}.
\end{equation}
If take $x$ in (\ref{tm1}) to be point in $T\subseteq U_{\sigma}\cap U_{\tau}$, then it is easy to see over the open subset $T$ of $U_{\sigma}\cap U_{\tau}$, the transition matrix $f_{\sigma\tau}$
can be written in the form of (\ref{tm}). Thus over $U_{\sigma}\cap U_{\tau}$, $f_{\sigma\tau}$ can be written in the form (\ref{tm}).

Now we fix a point $x_0$ in $T$, then one sees easily the transition matrices $\{f_{\sigma\tau}(x_0)\}$ satisfies the conditions (2) and (3) hence the proposition is proved.

\end{proof}

\subsubsection{Global Sections of a Toric Bundle}\hfill

\smallskip Suppose $X$ is a complete toric variety, $\mathcal{E}$ be a toric bundle on $X$. Then there is an induced action of $T$ on $H^0(X,\mathcal{E})$. Let $H^0(X,\mathcal{E})_{\chi}$ be the $\chi$-isotropy component of $H^0(X,\mathcal{E})$, then we have the following
\begin{pro}\label{pgs} Let $X(\Sigma)$ be a smooth complete toric variety and $\mathcal{E}$ be a toric bundle on $X$, then
\begin{equation} H^0(X,\mathcal{E})_{\chi}=\bigcap_{\alpha\in|\Sigma|}E^{\alpha}(\langle\chi,\alpha\rangle).
\end{equation}
\end{pro}

\smallskip Assuming the conditions in the above proposition, next we consider the image of a global section $s\in H^0(X,\mathcal{E})_{\chi}$ under the restriction map $\rho_{\sigma}:\Gamma(X,\mathcal{E})\rightarrow\Gamma(U_{\sigma},\mathcal{E})$, where $U_{\sigma}$ is an open affine sub toric variety of $X$.

Since $X$ is smooth and complete, it can be covered by sub toric varieties isomorphic to $\mathbb{A}^n$ with $n=\dim X$. Let $U_{\sigma}$ be such a sub toric variety, then $\mathcal{E}|_{U_{\sigma}}\cong\bigoplus_{1\leq i\leq r}\chi_i\mathcal{O}_{U_{\sigma}}$ and $\{e^i_{\sigma}\}$ be the associated eigen-basis of $\mathcal{E}|_{U_{\sigma}}$ as described in (\ref{dec0}). Then an element of $\Gamma(U_{\sigma},\mathcal{E})$ can be written as
$$v=\sum_{1\leq i\leq r}a_ie^i_{\sigma},\ \ \ a_i\in k[u_1,\cdots,u_n],$$
where $u_i\in \widehat{T}$ satisfying $\langle u_i,\sigma_j\rangle=\delta_{ij}$ for $\{\sigma_1,\cdots,\sigma_n\}=|\sigma|$.

Now we investigate the action of $T$ on $v$. Firstly, the action of $t=(t_1,\cdots,t_n)\in T\cong \mathbb{G}^n_{\mathbf{m}}$ on a monomial $b=u_1^{i_1}\cdots u_n^{i_n}$ is given by
\begin{equation}\label{chneg}t.b=\chi_b(t)b,\ \ \ \chi_b(t_j)=-i_j.
\end{equation}
Let $a_i=\sum_j c_{ij}b_{ij}$, where $c_{ij}\in k$, $b_{ij}\in k[u_1,\cdots,u_n]$ is a monomial, then the action of $t\in T$ on $v$ is given by
\begin{equation}\label{chneg1}t.v=\sum_i\sum_jc_{ij}\chi_{b_{ij}}(t)b_{ij}\chi^i_{\sigma}(t)e^i_{\sigma}\end{equation}

If $v=\rho_{\sigma}(s)$ for some $s\in H^0(X,\mathcal{E})_{\chi}$, then $t.e=\chi(t)e$ and one sees easily $a_i$ is a monomial for $1\leq i\leq r$. Then by (\ref{chneg})and (\ref{chneg1}), we have $\chi_{a_i}=\chi-\chi^i_{\sigma}$ and $a_i=c_i\prod_iu_i^{\chi^i_{\sigma}(\sigma_i)-\chi(\sigma_i)}$ for some $c_i\in k$.

\subsubsection{Examples}\hfill

\begin{exam}The cotangent and tangent bundle of a smooth toric variety.
\end{exam}
Given a smooth toric variety $X=X(\Sigma)$ of dimension $n$, the cotangent bundle $\Omega^1_X$ is a toric bundle on $X$. The associated Klyachko data is given by
\begin{equation}\label{cb}
E^{\alpha}(i)=
\begin{cases}
\Omega=\widehat{T}\otimes k, &\text{if $i\leq-1$;}\\
\{\omega\in\Omega\,|\,\langle\omega,\alpha\rangle=0\}, &\text{if $i=0$;}\\
0, &\text{if $i>0$.}
\end{cases}
\end{equation}

It suffices to prove (\ref{cb}) for affine smooth toric varieties. In this case, $X$ is isomorphic to $\mathbb{A}^d\times \mathbb{G}_{\mathbf{m}}^{n-d}$ and can be realized as the affine toric variety associated to the convex cone generated by the vectors $e_1,\cdots,e_d$ in $\widehat{T}^0\otimes_{\mathds{Z}}\mathds{R}$ if $d\geq 1$ and by origin if $d=0$.

Suppose the immersion $T\hookrightarrow X$ is given by
$$k[s_1,\cdots,s_d,s_{d+1},s_{d+1}^{-1},\cdots, s_n, s_n^{-1}]\hookrightarrow k[s_1,s_1^{-1},\cdots,s_n,s_n^{-1}].$$

Let $t=(t_1,\cdots, t_n)$be an element of $T\cong\mathbb{G}^n_{\mathbf{m}}$, then the action of $t$ on $\mathcal{O}_X$ is given by
\begin{equation}t.s_i=t_i^{-1}s_i,\;\;\; 1\leq i\leq n
\end{equation}
Then the action of $t$ on $ds_i$ is given by
\begin{equation}t.ds_i=t_i^{-1}ds_i,\;\;\; 1\leq i\leq n
\end{equation}
\noindent hence the Klyachko data of $\Omega^1_X$ given by (\ref{cb}) follows from (\ref{eri}).

\smallskip The tangent bundle $\mathcal{T}_X$ is also a toric bundle on $X$ and over any open affine sub toric variety of $X$ the associated
representation of $\mathcal{T}_X$ is the duality of the representation associated to $\Omega^1_X$. Thus the Klyachko data of $\mathcal{T}_X$ is given
by
\begin{equation}
E^{\alpha}(i)=
\begin{cases}
\mathscr{T}=\widehat{T}^0\otimes k, &\text{if $i\leq0$;}\\
k\alpha, &\text{if $i=1$;}\\
0, &\text{if $i>1$.}
\end{cases}
\end{equation}

\begin{exam}The symmetric product of a toric bundle.
\end{exam}
Let $\mathcal{E}$ be a toric bundle on a toric variety $X$ whose Klyachko data is given by $\{E^{\alpha}(i)\}_{\alpha\in|\Sigma|}$, $i\in\mathds{Z}$ as in
theorem \ref{kl}. As in (\ref{dec0}), we have the following decomposition
$$\mathcal{E}|_{U_{\sigma}}\cong\bigoplus_{\chi^i_{\sigma}\in\Xi_{\sigma}}\chi^i_{\sigma}\mathcal{O}_{U_{\sigma}},$$

\noindent where $\Xi_{\sigma}\subseteq\widehat{T}$ is the set of eigen-characters (counted with multiplicity). Then a decomposition of $S^m\mathcal{E}|_{U_{\sigma}}$ is given by
\begin{equation}\label{decse}
S^m\mathcal{E}|_{U_{\sigma}}\cong\bigoplus_{\chi^{i_1}_{\sigma},\cdots,\chi^{i_m}_{\sigma}\in\,\Xi_{\sigma}}(\chi_{\sigma}^{i_1}+\cdots+\chi_{\sigma}^{i_m})\mathcal{O}
_{U_{\sigma}}
\end{equation}
We denote by $E$ be the fiber of $\mathcal{E}$ over a point $x_{\sigma}$ in the unique closed orbit in $U_{\sigma}$, then the image be of the eigen-basis $\{e^i_{\sigma}\}$ are the eigenvector of the action of $T$ on $E$. We still use $\{e^i_{\sigma}\}$ to denote their images in $E$, then the Klyachko data of the toric bundle $S^m\mathcal{E}$ is given by
\begin{equation}\label{klsme}
(S^mE)^{\alpha}(i)\cong\bigoplus_{\substack{\chi^{i_1}_{\sigma},\cdots,\chi^{i_m}_{\sigma}\in\Xi_{\sigma},\,i_1\leq\cdots\leq i_m;\\
\langle\chi^{i_1}_{\sigma},\alpha\rangle
+\cdots+\langle\chi^{i_m}_{\sigma},\alpha\rangle\geq i}}ke^{i_1}_{\sigma}\cdots e^{i_m}_{\sigma}, \ \ \ \alpha\in|\sigma|.
\end{equation}

\begin{exam}The determinant of a toric bundle.
\end{exam}
Let $\mathcal{E}$ be a toric bundle on a toric variety $X$ whose Klyachko data is given by $\{E^{\alpha}(i)\}_{\alpha\in|\Sigma|}$, $i\in \mathds{Z}$ as in
theorem \ref{kl}. Then the Klyachko data associated to the line bundle $\det \mathcal{E}$ is given by the function
\begin{equation}\label{decdet}
\rho\mapsto\sum_{i\in I^{\alpha}(\mathcal{E})}(\mathrm{dim}E^{\alpha}(i)/\mathrm{dim}E^{\alpha}(i+1)i).
\end{equation}
%\begin{rem}\hfill
%\begin{enumerate}
%\item If $\sigma$ is a maximal cone in $\Sigma$, then $\chi^i_{\sigma}$, $1\leq i\leq n$ are uniquely determined by $\mathcal{E}$ up to permutation of the subscripts.
%\item If $\sigma$ is not a maximal cone in $\Sigma$, the characters $\chi^i_{\sigma}$, $1\leq i\leq n$ are not uniquely determined by the toric bundle $\mathcal{E}$, as is mentioned in remark (\ref{uc}).
%\item Let $P_{\sigma\sigma'}=(p_{ij})$, if $p_{ij}\neq0$, then we have $\langle(\chi^i_{\sigma})^{-1}\chi^j_{\sigma'},\rho\rangle\geq 0$ for all $\rho\in\sigma\cap\sigma'(1)$. However the characters $(\chi^i_{\sigma})^{-1}$ and $\chi^j_{\sigma'}$ themselves may not be defined over $U_{\sigma}\cap U_{\sigma'}$.
%\end{enumerate}
%\end{rem}

\section{Proof of the Theorems}

\smallskip Based on the criteria given in the preceding section, next we will prove our theorems.

\begin{proof}[Proof of Theorem \ref{mt}] By proposition \ref{ka}, $\mathcal{E}$ can be represented by an element of $H^1(X,\mathcal{GL}_2)$ whose transition matrices are of the form (\ref{tm}). In order to prove $\mathbb{V}(\mathcal{E})$ is compatibly \emph{F}-split with $X$, it suffices to prove for each $M_{\sigma\tau}, \sigma,\tau\in\Delta$ its entries satisfy the condition of Lemma \ref{le}.

\smallskip Let $M_{\sigma\tau}=\binom{a_{11}\,a_{12}}{a_{21}\,a_{22}}$, then by (\ref{tm}) we have $a_{ij}=\chi^i_{\sigma}p_{ij}(\chi^j_{\tau})^{-1}$, where $p_{ij}\in k$ is the $ij$-th entry of $P_{\sigma\tau}$. Then
$a_{11}^{s}a_{21}^{p-s}=p^s_{11}p^{p-s}_{21}(\chi^1_{\sigma})^s(\chi^2_{\sigma})^{p-s}(\chi^1_{\tau})^{-p}$. Note that though the element $(\chi^1_{\sigma})^s(\chi^2_{\sigma})^{p-s}$ may not be defined over $U_{\sigma}\cap U_{\tau}$, we have
\begin{equation}\varphi(a_{11}^{s}a_{21}^{p-s})=
\varphi(p^s_{11}p^{p-s}_{21}(\chi^1_{\sigma})^s(\chi^2_{\sigma})^{p-s}(\chi^1_{\tau})^{-p})=p^s_{11}p^{p-s}_{21}
(\chi^1_{\tau})^{-1}\varphi((\chi^1_{\sigma})^s(\chi^2_{\sigma})^{p-s})
\end{equation}

\noindent over the open subset of $U_{\sigma}\cap U_{\tau}$ where $\chi^1_{\sigma}$ is invertible. Therefore, we only need to show $\varphi((\chi^1_{\sigma})^s(\chi^2_{\sigma})^{p-s})=0$ for all $1\leq s\leq p-1$.

\smallskip Since the toric variety $X$ is smooth, the affine toric variety $U_{\sigma}\cap U_{\tau}$ is isomorphic to $\mathbb{A}^d\times\mathbb{G}^{n-d}_m$ for some integer $0\leq d\leq n$. To be more precise, let $\alpha_i, 1\leq i\leq d$
be the primitive generators of $\sigma$ which constitutes part of a $\mathds{Z}$-basis $\{\alpha_1,\alpha_2,\cdots,\alpha_n\}$ of $\widehat{T}^0$. We assume further $\langle\rho_i,x_j\rangle=\delta_{ij}$. Then we only need to show the monomial
$(\chi^1_{\sigma})^s(\chi^2_{\sigma})^{p-s}$ is not a $p$-th power in $K(X)\cong k(x_1,\cdots, x_n)$.

\smallskip Let $\chi^1_{\sigma}=x^{i_1}_1\cdots x^{i_n}_n$, $\chi^2_{\sigma}=x^{j_1}_1\cdots x^{j_n}_n$, then by condition $-p<i_k-j_k< p$ and $i_k\neq j_k$ for all $1\leq k\leq n$. Therefore $\frac{\chi^1_{\sigma}}{\chi^2_{\sigma}}$ is not a $p$-th power over $U_{\sigma}$, which implies $(\chi^1_{\sigma})^{p-s}(\chi^2_{\sigma})^s$ is not a $p$-th power.

\end{proof}

\begin{corollary}The tangent bundle and cotangent bundle of a smooth toric surface F-split compatibly with the toric surface.
\end{corollary}

\begin{proof}
By \cite[Example 2.3.5]{Klyachko1990}, the tangent bundles and cotangent bundles of a smooth toric surface always satisfy the condition of the theorem.

\end{proof}

Before proving theorem \ref{mt2}, we will investigate the Frobenius splitting of a toric variety in a little more detail. Since all toric varieties are \emph{F}-split, by our discussion in the previous section there exists a global section of $\omega^{1-p}_X$ for any smooth complete toric variety $X$ such that its image under $\iota$ (\ref{iota}) defines a Frobenius splitting of $X$. On the other hand, we have the following decomposition
\begin{equation}
H^0(X,\omega_X^{1-p})\label{omgdc}\cong\bigoplus_{\chi\in\widehat{T}}H^0(X,\omega_X^{1-p})_{\chi},
\end{equation}
\noindent where $H^0(X,\omega_X^{1-p})_{\chi}$ is the $\chi$-isotropy subspace of $H^0(X,\omega_X^{1-p})$.
\begin{lemma}\label{le4} Let $X=X(\Sigma)$ be a smooth complete toric variety, then the character $\mu\in\widehat{T}$ defined by $\langle\mu,\sigma\rangle=0$ for all $\sigma\in\Sigma$ is the unique character $\chi\in\widehat{T}$ such that there exists a nonzero element of $H^0(X,\omega^{1-p}_X)_{\chi}$ whose image under $\iota$
defines a Frobenius splitting of $X$.
\end{lemma}
\begin{proof}We first prove the existence of a character $\chi\in\widehat{T}$ such that a nonzero vector in $H^0(X,\omega^{1-p}_{X})$ defines a Frobenius splitting of $X$. By the main result of \cite{thomsen2000},
$F_{X*}\mathcal{O}_X$ splits into direct sum of line bundles as follows $$F_{X*}\mathcal{O}_X\cong\mathcal{O}_X\oplus\bigoplus_{1\leq i\leq p^n-1}\mathcal{L}_i,$$

\noindent where $n=\dim X$. It is easy to see $H^0(X,\mathcal{L}_i)=0$ for all $1\leq i\leq p^n-1$. Let $\varphi_0:F_{X*}\mathcal{O}_X\rightarrow \mathcal{O}_X$ be the morphism whose restriction of $\mathcal{O}_X$ is identity and to $\mathcal{L}_i$ is $0$ for all $1\leq i\leq p^n-1$. Furthermore, let $H$ be the subspace of $Hom(F_{X*}\mathcal{O}_X,\mathcal{O}_X)$ consisting of morphisms whose restrictions to $\mathcal{O}_X$ is zero. Then one can find vectors $\varphi_i\in H,1\leq i\leq m$, such that $\{\varphi_i\}_{0\leq i\leq m}$ form a basis of $Hom(F_{X*}\mathcal{O}_X,\mathcal{O}_X)$.

On the other hand, we have the following decomposition
\begin{equation}
Hom(F_{X*}\mathcal{O}_X,\mathcal{O}_X)\cong H^0(X,\omega_X^{1-p})\cong\bigoplus_{\chi\in\widehat{T}}H^0(X,\omega_X^{1-p})_{\chi}.
\end{equation}
Thus there exists a character $\chi\in\widehat{T}$ such that $H^0(X,\omega_X^{1-p})_{\chi}\neq0$ and a nonzero vector $v\in H^0(X,\omega_X^{1-p})_{\chi}\neq0$ such that $v=\sum_ia_i\varphi_i$, $a_0\neq0$. Then we can find a vector $v$ such that the coefficient $a_0$ is exactly $1$, then $v$ defines a Frobenius splitting of $X$.

\medskip Next we prove zero character is the only one with the property described in the lemma. Let $\sigma\in\Sigma$ be a cone with maximal dimension, then the primitive vectors $\{\sigma_i\,|\,1\leq i\leq n\}=|\sigma|$ form a basis of $\widehat{T}^0\otimes_{\mathds{Z}}\mathds{R}$. Let $\chi\in\widehat{T}$ be a character defined by $\langle\chi,\sigma_i\rangle=l_i$ and assume $v\in H^0(X,\omega^{1-p}_X)_{\chi}$ defines a Frobenius splitting of $X$. By our discussion in section 2.2.3, the image of $v$ under the restriction map $\rho_{\sigma}:H^0(X,\omega^{1-p}_X)\rightarrow \Gamma(U_{\sigma},\omega^{1-p}_X)$ is of the form $$cu_1^{p-1-l_1}\cdots u_n^{p-1-l_n}e_{\sigma},$$

\noindent where $c\in k$ and $p-1-l_1,\cdots, p-1-l_n$ are necessarily nonnegative.

On the other hand, the section $e_{\sigma}\in\Gamma(U_{\sigma},\omega_X^{1-p})\cong\mathscr{H}om(\omega_X^p,\omega_X)|_{U_{\sigma}}$ is defined by sending the section $\omega_{\sigma}^{\otimes p}$ to $\omega_{\sigma}$, where $\omega_{\sigma}=du_1\wedge\cdots\wedge du_n$. Thus by Lemma \ref{le3}, $\iota(\rho_{\sigma}(v))(1)$ is nonzero iff $l_1=\cdots=l_n=0$ hence the lemma is proved.

\end{proof}

The above lemma applies to projective spaces, and we are lead to the following definition of an \emph{FS}-vector.

\begin{definition} Let $E$ be a vector space of dimension $r$ over $k$, then a vector $v$ in $S^{r(p-1)}E$ is called an FS-vector if under a basis $\{e^i\}_{1\leq i\leq r}$ of $E$, $v=ce_1^{p-1}\cdots e_r^{p-1}$ for some $c\in k^*$.
\end{definition}

\begin{proof}[Proof of Theorem 2] We first prove the sufficiency of the condition in the theorem. By criteria B in the previous section, to prove $Y=\mathbb{P}(\mathcal{E})$ is \emph{F}-split, it suffices to show the map $\pi$ (\ref{gs}) is surjective. Since $H^0(X,\omega^{1-p}_X)$ has the decomposition (\ref{omgdc}) and a vector of $H^0(X,\omega^{1-p}_X)_{\mu}$ defines a Froenius splitting of $X$, it suffices to prove $H^0(X,\omega^{1-p}_X)_{\mu}$ is contained in the image of $\pi$.

On the other hand, note we have
$$f_*\omega_Y^{1-p}\cong S^{r(p-1)}\mathcal{E}\otimes\det\mathcal{E}^{1-p}\otimes\omega_X^{1-p},$$
whence an induced morphism
\begin{equation}\label{pi}H^0(X,S^{r(p-1)}\mathcal{E}\otimes\det\mathcal{E}^{1-p}\otimes\omega_X^{1-p})\rightarrow H^0(X,\omega_X^{1-p}),\end{equation}
which is still denoted by $\pi$.
\noindent Then by lemma \ref{le4} it suffices to prove the image of an \emph{FS}-vector in
$H^0(X,S^{n(p-1)}\mathcal{E}\otimes\det\mathcal{E}^{1-p}\otimes\omega^{1-p}_X)_{\mu}$ under $\pi$ falls in
$H^0(X,\omega_X^{1-p})_{\mu}$ and is nonzero.

\smallskip First we prove the map (\ref{pi}) is equivariant. Let $s$ be an element of the left side, then we need to show $t.\pi(s)=\pi(t.s)$ for $t\in T$. We can further reduce to prove this claim in the case when $s$ is an eigenvector, i.e. $t.s=\chi(t)s$ for some $\chi\in\widehat{T}$ and $t\in T$.

\smallskip Let $U_{\sigma}$ be an open affine sub toric variety of $X$, then we can assume $U_{\sigma}$ isomorphic to $\mathbb{A}^n$ since $X$ is smooth. Let $\{e^i_{\sigma}\}_{1\leq i\leq r}$ be an eigenbasis of $\mathcal{E}|_{U_{\sigma}}$, $\{\chi^i_{\sigma}\}_{1\leq i\leq r}$ be the corresponding eigen-characters. By trivializing the line bundle $\det\mathcal{E}^{1-p}\otimes\omega_X^{1-p}$, we can use $\{(e_{\sigma}^1)^{i_1}\cdots (e_{\sigma}^r)^{i_r}\}$ to denote the basis of $S^{r(p-1)}\mathcal{E}\otimes\det\mathcal{E}^{1-p}\otimes\omega^{1-p}_X$, where $i_1,\cdots,i_r$ are nonnegative integers and $i_1+\cdots+i_r=r(p-1)$. For simplicity, $(e_{\sigma}^1)^{i_1}\cdots (e_{\sigma}^r)^{i_r}$ will be denoted by $e_{\sigma}^I$ with $I=(i_1,\cdots,i_r)$ in the sequel. By (\ref{decse}) and (\ref{decdet}), the eigen-characters corresponding to $e^I_{\sigma}$ is given by

$$\chi^I_{\sigma}=\sum_{1\leq j\leq n}i_j\chi^j_{\sigma}-\sum_{1\leq j\leq n}(p-1)\chi^j_{\sigma}+\chi_{p-1},$$

\noindent where $\chi_{p-1}\in \widehat{T}$ is the character defined by $\langle\chi_{p-1},\sigma_i\rangle=p-1$ for all $1\leq i\leq n$.

By our discussion in section 2.2.3, the image of $s$ under the restriction map $\rho_{\sigma}:H^0(S^{r(p-1)}\mathcal{E}\otimes\det\mathcal{E}^{1-p}\otimes\omega_X^{1-p})\rightarrow
\Gamma(U_{\sigma},S^{r(p-1)}\mathcal{E}\otimes\det\mathcal{E}^{1-p}\otimes\omega_X^{1-p})$ can be written as linear combination of terms of the following form
$$a_Ie_I\otimes\omega_{\sigma}^{1-p},$$

\noindent where $a_I=u_1^{\chi_{\sigma}^I(\sigma_1)-\chi(\sigma_1)}\cdots u_n^{\chi_{\sigma}^I(\sigma_n)-\chi(\sigma_n)}$. Then the image of $\rho_{\sigma}(s)$ under $\pi$ is
\begin{equation}\label{pis}
\begin{cases}u_1^{p-1-\chi(\sigma_1)}\cdots u_n^{p-1-\chi(\sigma_n)}\omega_{\sigma}^{1-p}, & \text{if $i_j=p-1$ for all $1\leq j\leq n$};\\
             0, & \text{otherwise}.
\end{cases}
\end{equation}

\noindent Then one can see easily $t(\pi(\rho_{\sigma}(s)))=\chi(t)\pi(\rho_{\sigma}(s))$, so the morphism $\pi$ is $T$-equivariant.

Given a global section $s\in H^0(S^{r(p-1)}\mathcal{E}\otimes\det\mathcal{E}^{1-p}\otimes\omega^{1-p}_X)_{\mu}$ that can be represented by an  \emph{FS}-vector, one sees easily by Lemma \ref{tr} and \ref{le3}, the image $\pi(s)$ defines a Frobenius splitting of $X$ hence the sufficiency of the condition is proved.

Now we prove the necessity of the condition. Suppose there exists a global section $s\in H^0(S^{r(p-1)}\mathcal{E}\otimes\det\mathcal{E}^{1-p}\otimes\omega^{1-p}_X)$ such that $\iota(\pi(s))$ defines a Frobenius splitting of $X$, then by the fact that $\pi$ is $T$-equivariant and lemma \ref{le4}, the image of $H^0(S^{r(p-1)}\mathcal{E}\otimes\det\mathcal{E}^{1-p}\otimes\omega^{1-p}_X)_{\mu}$ under $\pi$ is nonzero. Indeed, since the dimension $H^0(X,\omega^{1-p}_X)_{\mu}$ is 1, one can choose a global section $s\in H^0(S^{r(p-1)}\mathcal{E}\otimes\det\mathcal{E}^{1-p}\otimes\omega^{1-p}_X)_{\mu}$ such that $\iota(\pi(s))$ defines a Frobenius splitting of $X$. On the other hand, by our proof of the $T$-equivariance of $\pi$, the image $\pi(s)$ is nonzero only if $s$ can be represented by an $FS$-vector, hence the necessity and the theorem is proved.

\end{proof}

\begin{corollary} Let $X=X(\Sigma)$ be a smooth complete toric variety, then $\mathbb{P}\mathcal{T}_X$ is Frobenius split.
\end{corollary}

\begin{proof}Let $\{\sigma_i\}_{1\leq i\leq n}$ be the ray generators of a cone $\sigma\in\Sigma$ with maximal dimension, then by proposition \ref{pgs} and (\ref{klsme}), $\sigma_1^{p-1}\cdots\sigma_n^{p-1}$ is an \emph{FS}-vector of $H^0(S^{n(p-1)}\mathcal{T}_X)_{\mu}$.
Thus by theorem 2, $\mathbb{P}\mathcal{T}_X$ is \emph{F}-split .

\end{proof}
\bibliographystyle{amsplain}
\bibliography{xhBibTex}

\providecommand{\bysame}{\leavevmode\hbox to3em{\hrulefill}\thinspace}
\providecommand{\MR}{\relax\ifhmode\unskip\space\fi MR }
\providecommand{\MRhref}[2]{%
  \href{http://www.ams.org/mathscinet-getitem?mr=#1}{#2}
}
\providecommand{\href}[2]{#2}
\begin{thebibliography}{1}

\bibitem{kumarbrion}
M.~Brion and S.~Kumar, \emph{Frobenius splitting methods in geometry and
  representation theory}, Progress in mathematics, vol. 231, Birkh\"{a}user,
  2005.

\bibitem{hartshorne1966}
R.~Hartshorne, \emph{Residues and duality}, Lecture Notes in Math., vol.~20,
  Springer-Verlag, Berlin-Heidelberg-New York, 1966.

\bibitem{hartshorne1977algebraic}
\bysame, \emph{Algebraic geometry}, vol.~52, Springer-Verlag, 1977.

\bibitem{hmp2010}
M.~Hering, M.~Musta\c{t}\u{a}, and S.~Payne, \emph{Posivity properties of toric
  vector bundles}, Annales de L'institut Fourier \textbf{60} (2010), no.~2,
  607--640.

\bibitem{Kaneyama1975}
T.~Kaneyama, \emph{On equivariant vector bundle on an almost homogeneous
  variety}, Nagoya Math. J. \textbf{57} (1975), no.~2, 65--86.

\bibitem{Klyachko1990}
A.A. Klyachko, \emph{Equivariant bundles on toral varieties}, Math. USSR
  Izvestiya \textbf{35} (1990), no.~2, 337--375.

\bibitem{thomsen2000}
J.~F. Thomsen, \emph{Frobenius direct images of line bundles on toric
  varieties}, Journal of Algebra \textbf{226} (2000), no.~2, 865--874.

\end{thebibliography}

\end{document}